\documentclass[11pt,onecolumn]{article}
\usepackage[T1]{fontenc}
\usepackage[english]{babel}
\usepackage{lmodern}
\usepackage{microtype}
\usepackage[onehalfspacing]{setspace}
\usepackage[bottom=2cm]{geometry}
\usepackage{amsmath,amssymb,amsfonts,amsthm,mathtools}
\usepackage{amsrefs}
\usepackage{enumitem}
\usepackage{graphicx}
\usepackage{tikz-cd}
\usetikzlibrary{babel}
\usepackage{hyperref}
\usepackage{cleveref}

\tikzcdset{scale cd/.style={every label/.append style={scale=#1},
    cells={nodes={scale=#1}}}}

\newcommand{\N}{\mathbb{N}}

\newcommand{\C}{\mathbb{C}}

\newcommand{\A}{\mathcal{A}}

\newcommand{\R}{\mathbb{R}}

\setlist[enumerate]{label=(\roman*),itemsep=3pt}

\theoremstyle{plain}
\newtheorem{theorem}{Theorem}[section]
\newtheorem{lemma}[theorem]{Lemma}
\newtheorem{corollary}[theorem]{Corollary}
\newtheorem{proposition}[theorem]{Proposition}

\theoremstyle{definition}
\newtheorem{definition}[theorem]{Definition}

\theoremstyle{remark}
\newtheorem{remark}[theorem]{Remark}

\title{Quantifier Elimination and Invariant Theory: Applications to Quaternions, Octonions, and Other Algebras}
\author{Maximilian Illmer}
\date{\today}

\begin{document}
\maketitle

\begin{abstract}
\noindent
We build on our previous paper \cite{constructive} by using the general method introduced there in conjunction with invariant theory. This yields quantifier elimination results for the classical quaternions and octonions, as well as criteria for other classes of finite-dimensional algebras over real closed and algebraically closed fields. In particular, the first two applications answer an open question posed in~\cite{savi}.
\end{abstract}

\section{Introduction and Preliminaries}

\subsection{Introduction}

Quantifier elimination is a central topic in model theory, with deep connections to algebra, geometry, and computation, especially in the real case \cites{basu,boch,pd}. Classical results establish quantifier elimination for the theories of algebraically closed fields and real closed fields in the languages $\mathcal{F}:=(+,-,\times,0,1)$ and $(\mathcal{F},\leq)$, respectively.

In contrast, much less is known about quantifier elimination for noncommutative algebras. Even for natural and well-studied structures such as the classical quaternions $\mathbb{H}(k)$ or octonions $\mathbb{O}(k)$, for $k$ real closed, the model-theoretic behavior is considerably more subtle. In particular, it was recently shown in \cite{savi} that their first-order theories in the pure ring language $\mathcal{F}$ do not admit quantifier elimination, and the question of which language extensions yield quantifier elimination was left open.

In this paper, we develop a systematic method for obtaining quantifier elimination for finite-dimensional algebras via invariant theory. Building on the constructive framework introduced in \cite{constructive}, we show that the existence of sufficiently expressive invariants for diagonal automorphism actions allows one to transfer quantifier elimination from the base field to the algebra.

As applications, we prove quantifier elimination for the classical quaternions and octonions over real closed fields in a natural language expansion by conjugation $\bar \ $ and the scalar order relation $\leq$. More generally, we obtain criteria for quantifier-elimination for finite-dimensional algebras $\A$ over real closed and algebraically closed fields in terms of their invariant algebras $k[\A^m]^G$, for $G\leq\mathrm{Aut}_k(\A)$, and the expressibility of their generators in suitable language extensions of $\mathcal{F}$.

The method is constructive whenever quantifier elimination in the base field is constructive and our "transfer function" (see \Cref{theo}) is explicitly known. In particular, for real quaternions, eliminating $l$ quantified variables from a conjunction of $r$ atomic formulas in $l+m$ variables reduces to a quantifier-elimination problem over $k$ with $4(l+m)$ quantified scalar variables and $m+m^2+m^3$ free invariant variables. The resulting conjunction consists of $4r+m+m^2+m^3$ atomic formulas in $4(l+m)+m+m^2+m^3$ variables. For octonions, the analogous statement holds with $4$ replaced by $8$ and $m+m^2+m^3$ replaced by $o_m=|I(m)|=\mathcal{O}(m^4)$. These explicit bounds suggest that implementation may be feasible for formulas with favorable block structure, although the practical complexity of course depends on the quantifier-elimination procedure used over the base field (compare also to the complexity analysis in \cite{constructive}).

\subsection{Preliminaries from Algebra}

We briefly introduce real closed fields, algebraically closed fields, quaternions and octonions.\\

A field $k$ is \emph{algebraically closed} if every nonconstant polynomial over $k$ has a zero in $k$. The most prominent example is $k=\C$. Every field is contained in an algebraic extension that is algebraically closed and unique up to $k$-isomorphism, called its \textit{algebraic closure} $\bar k$.
\medskip
A field $k$ is \emph{real closed} if $-1$ is not a square in $k$ and the field extension $k(\sqrt{-1})$ is algebraically closed (in which case it is the algebraic closure of $k$). There are other characterizations \cite{pd}, but this one is the easiest to state. Here, $\R$ is the most prominent example. Real closed fields have a unique field ordering, where an element is nonnegative if and only if it is a square. Every ordered field is contained in a real closed field.

If $k$ is real closed, then on $k(\sqrt{-1})$ one obtains an \emph{involution}, just as complex conjugation on $\C$, by multiplying the coefficient of $\sqrt{-1}$ by $-1$. As we will see below, this also applies to the quaternions $\mathbb{H}$ and octonions $\mathbb{O}$ over $k$.

For the remainder of this subsection, $k$ denotes a real closed field.

\medskip

The (non-split, classical) \emph{quaternions over $k$}, denoted $\mathbb{H}(k)$, form the $k$-algebra generated by elements $\mathbf{i},\mathbf{j}$ satisfying
\[
\mathbf{i}^2=\mathbf{j}^2=-1,
\qquad
\mathbf{i}\mathbf{j}=-\mathbf{j}\mathbf{i}.
\]
Setting $\mathbf{k}=\mathbf{i}\mathbf{j}$, one obtains
\[
\mathbf{i}^2=\mathbf{j}^2=\mathbf{k}^2=-1,
\qquad
\mathbf{i}\mathbf{j}=\mathbf{k}=-\mathbf{j}\mathbf{i}.
\]
Every element of $\mathbb{H}(k)$ has a unique expression
\[
x=a+b\mathbf{i}+c\mathbf{j}+d\mathbf{k},
\qquad a,b,c,d\in k,
\]
i.e.\ the underlying vector space is $k^4$.

In this vector space representation we introduce the notation $\operatorname{Re}(x):=a$ and $\operatorname{Im}(x):=(b,c,d)$.

$\mathbb{H}(k)$ is central simple and associative over $k$. It carries a standard involution (conjugation)
\[
\overline{x}=a-b\mathbf{i}-c\mathbf{j}-d\mathbf{k},
\]
whose associated quadratic form
\[
N(x)=x\overline{x}=a^2+b^2+c^2+d^2
\]

is positive definite and multiplicative. In particular, $\mathbb{H}(k)$ is, up to isomorphism, the unique $4$-dimensional central associative division algebra over $k$ and a composition algebra.

\medskip

The (classical, non-split) \emph{octonions over $k$}, denoted $\mathbb{O}(k)$, are obtained from $\mathbb{H}(k)$ by the Cayley--Dickson doubling process. As a $k$-vector space,
\[
\mathbb{O}(k)=\mathbb{H}(k)\oplus \mathbb{H}(k)\boldsymbol{\ell},
\qquad \boldsymbol{\ell}^2=-1,
\]
with multiplication defined by
\[
(a+b\boldsymbol{\ell})(c+d\boldsymbol{\ell})
=
(ac-\overline{d}b)+(da+b\overline{c})\boldsymbol{\ell},
\qquad a,b,c,d\in\mathbb{H}(k).
\]

The algebra $\mathbb{O}(k)$ is $8$-dimensional, noncommutative, nonassociative but \emph{alternative}. Just as for quaternions, we use the vector space structure to define, for $x=a+b\boldsymbol{\ell}$,
\[
\operatorname{Re}(x):=\operatorname{Re}(a),
\qquad
\operatorname{Im}(x):=(\operatorname{Im}(a),\operatorname{Re}(b),\operatorname{Im}(b)).
\]
In some literature (for example \cites{schwarzoct, alginv}), $2\operatorname{Re}$ is also denoted as the \emph{trace} $\operatorname{tr}$.

Conjugation extends by
\[
\overline{a+b\boldsymbol{\ell}}=\overline{a}-b\boldsymbol{\ell},
\]
and the norm
\[
N(x)=x\overline{x}
\]
remains multiplicative and positive definite. Thus $\mathbb{O}(k)$ is, up to isomorphism, the unique $8$-dimensional composition division algebra over $k$.

\subsection{Preliminaries from Logic}

The model theory, used throughout this paper will be very light. We will work in the classical setup of first-order logic and model theory, as described for example in \cites{cha, hod, pre}. We denote formal languages by calligraphic letters such as $\mathcal{L}$.
\medskip
A language can be extended by an additional symbol $S$; we denote this by $(\mathcal{L},S)$. In order to speak of rings, a language $\mathcal{L}$ needs two binary function symbols $+,\cdot$, a unary function symbol $-$, and two constant symbols $0,1$. With these symbols one can write down the ring axioms and the axioms of algebraically closed or real closed fields, for example. Throughout this paper, we denote this specific language by $\mathcal{F}$. The complex numbers $\C$ form an obvious $\mathcal{F}$-structure, with the usual interpretations.
\medskip
If we want to talk about ordered rings, we can extend the language by the relation symbol $\leq$. An obvious $(\mathcal{F},\leq)$-structure is $\R$ with the usual interpretations. Note that for real closed fields, the extension by $\leq$ is to some extent redundant, since nonnegative elements are squares and nonnegativity can thus be expressed in $\mathcal{F}$ alone. However, this uses a quantifier, and to obtain quantifier elimination one needs to include the relation symbol $\leq$.

\medskip

To denote a \textit{formula} we will use small greek letters like $\phi$ and to emphasize the free variables $x_1,\dots,x_m$ of a formula $\phi$, we sometimes write $\phi(x_1,\dots,x_m)$. Given such a formula and a structure $A$, $\phi$ defines the subset of $A^m$ of all $m$-tuples that satisfy the formula, and we write $\phi(A)$ for this subset. Such sets are called \emph{definable} (some other also specify this as $\emptyset$-definable).

\medskip

We say an $\mathcal{L}$-structure $A$ admits \emph{quantifier elimination} if every definable set is also definable without quantifiers. More precisely, for every $\mathcal{L}$-formula $\phi(x_1,\dots,x_m)$ there exists a quantifier-free $\mathcal{L}$-formula $\psi(x_1,\dots,x_m)$ with
\[
\phi(A)=\psi(A).
\]

\medskip

It is well known that algebraically closed fields, i.e.\ $\mathcal{F}$-structures satisfying the axioms of algebraically closed fields, and real closed fields, i.e.\ $(\mathcal{F},\leq)$-structures satisfying the axioms of real closed fields, admit quantifier elimination. In fact, a stronger result holds: quantifiers can be removed independently of the actual field (that is, the \emph{theory} of algebraically closed or real closed fields admits quantifier elimination).

Both results can be proven by showing that the projection of a definable set is again definable. The definable sets in algebraically closed fields are precisely the constructible sets in the Zariski topology, while those in real closed fields are precisely the semialgebraic sets. Since projections are defined by formulas of the form $\exists x\,\phi(x,x_1,\dots,x_m)$, the general case can then be deduced inductively. In the case of algebraically closed fields, a constructive proof reduces to a simple application of the Euclidean algorithm. For real closed fields the result is less trivial and is the content of the Tarski–Seidenberg Theorem.  Also in this case constructive proofs exist (see for example \cites{basu, boch, sch}).

\medskip

Let $k$ be a field in the language $\mathcal{F}$. Any unital $k$-algebra can also be considered as a structure over the language $\mathcal{F}$ by interpreting the symbols $+,-,\times$ as the respective algebra addition, subtraction, and multiplication, and $1$ and $0$ as $1_\A$ and $0_\A$, respectively.

If $k$ is an ordered field in the language $(\mathcal{F},\leq)$, a finite-dimensional $k$-algebra can also be defined over this extended language by setting $a\leq b$ precisely when $a,b\in k1_\A$ and their corresponding scalars satisfy the order on $k$. In the pure ring language, only the elements $c1_\A$ for $c$ algebraic over the prime subfield are definable. This is why one sometimes sees the language extension $(\mathcal{F},(\lambda_c)_{c\in k})$, where $\lambda_c$ denotes scalar multiplication with $c$. Since our interest is not in axiomatizing the $k$-algebra models, but in interpreting fixed $k$-algebras with a minimal language, allowing for meaningful formulas, we will forego these additional symbols.

\medskip

If the center
\[
C(\A)=\{a\in\A:\ ax=xa,\ (a,x,y)=(x,a,y)=(x,y,a)=0
\text{ for all }x,y\in\A\},
\]
where $(x,y,z):=(xy)z-x(yz)$ denotes the associator,
equals $k1_\A$ we have an easy way to identify the base field $k$ as a definable subset of $\A$. However, even if $k1_\A$ is a strict subset of the center, we can (and will) use the natural isomorphism $\iota:k\to k\cdot 1_\A$ as an external identification, for example to transform functions $f:\A^m\to k^n$ into functions with values in $\A$.

\medskip

The algebras we will be most interested in are $\mathbb{H}(k)$ and $\mathbb{O}(k)$ for real closed $k$. In the recent paper \cite{savi}, general model-theoretic properties have been investigated for theories in the language $\mathcal{F}$ whose models are, up to isomorphism, precisely the algebras $\mathbb{H}(k)$ or $\mathbb{O}(k)$, respectively, as $k$ ranges over real closed fields. In particular, it was shown that these theories do not have quantifier elimination in $\mathcal{F}$, and the question for possible language extensions $\mathcal{F}'$ such that quantifier elimination holds, was posed.

We will interpret $\mathbb{H}(k)$ and $\mathbb{O}(k)$ as $(\mathcal{F},\leq,\bar{\ })$-structures, where $\leq$ is interpreted as described above for general $k$-algebras and $\bar{\ }$ as the conjugation, defined in our introductory section on quaternions and octonions. As it will turn out, $(\mathcal{F},\leq,\bar{\ })$ is a language extension allowing for quantifier elimination in both cases.

\medskip

Just as in our previous paper (although less explicitly stated there), the most important model-theoretic concept for proving this will be that of types and the separation thereof.

Let $A$ be a structure in the language $\mathcal{L}$. We say $\underline{a},\underline{b}\in A^m$ have the same $\emptyset$-type if for every $\mathcal{L}$-formula $\phi(x_1,\dots,x_m)$ one has
\[
\underline{a}\in \phi(A) \iff \underline{b}\in \phi(A) 
\]

\medskip

For a $\mathcal{L}$-structure $A$, a bijective function $f:A\to A$ is called an \emph{automorphism} if it preserves the respective interpretations of every constant-/function-/relation-symbol. The set of automorphisms forms a group with regards to composition, denoted $\mathrm{Aut}(A)_\mathcal{L}$. We will omit the subscript if the language $\mathcal{L}$ is clear from the context.

It is easy to see that the diagonal action of automorphisms preserves types.

\begin{theorem}\label{type}
Let $A$ be a structure in the language $\mathcal{L}$ and denote by $\mathrm{Aut}(A)$ its automorphism group. Then for $m$-tuples $\underline{a},\underline{b}\in A^m$ one has
\[
\underline{a}\in\mathrm{Orb}(\underline{b},\mathrm{Aut}(A))
\quad\Rightarrow\quad
\underline{a}\text{ and }\underline{b}\text{ are of the same }\emptyset\text{-type.}
\]
\end{theorem}

\subsection{Linear Algebraic Groups and Invariant Theory}

We briefly review the notions and results from invariant theory and the theory of linear algebraic groups needed in the sequel. A complete treatment is impossible here, the references we used include \cites{GIT, bredon, borel, lie, weyl, grot, sym}. The modern treatment of geometric invariant theory is heavily scheme-theoretic but for the purpose of our paper it is luckily possible to avoid the more technical details. We will state all the relevant results, such that they are readable without any background in algebraic geometry.
In that spirit we start with the, not maximally general, definition of a linear algebraic group.

\medskip

\begin{definition}
Let $k$ be an algebraically closed field and $G\leq\mathrm{GL}_n(k)$. We say that $G$ is a \emph{linear algebraic group} (over $k$) if it is Zariski closed in $\mathrm{GL}_n(k)$, i.e.\ if there exists an ideal $I$ of polynomial functions in the matrix entries such that
\[
G=\mathcal{V}(I)\subset\mathrm{GL}_n(k).
\]
The coordinate ring of $G$ is defined as
\[
k[G]=k[\mathrm{GL}_n]/\sqrt{I},
\]
where
\[
k[\mathrm{GL}_n]
=
k[(M_{ij})_{i,j\leq n},\text{det}^{-1}]
\]
and $\sqrt{I}=\mathcal{I}(G)$ denotes the ideal of polynomial functions vanishing on $G$.
\end{definition}
\medskip

If $l$ is a subfield of $k$ and $\mathcal{I}(G)$ is generated (as a $k$-ideal) by $\mathcal{I}(G)\cap l[\mathrm{GL}_n]$, we say that $G$ is \emph{defined over} $l$. In this case the group of $l$-rational points is
\[
G(l)=G\cap\mathrm{GL}_n(l).
\]
If $G$ is a linear algebraic group defined over a field $k$ and $K/k$ is a field extension, we denote its base change by $G_K$ and its group of $K$-rational points by $G(K)$. This is especially relevant when allowing for linear algebraic groups over arbitrary fields.\\
A more general notion is that of an $l$-form of $G$ (for details see \cite{borel} or \cite{sym} for the case $l=\R$).

\medskip

As an example, the group $\mathrm{GL}_n(\C)$ is defined over $\R$, its $\R$-rational points are $\mathrm{GL}_n(\R)$, yet it has many more $\R$-forms, for example the indefinite unitary groups $\mathrm{U}(p,q)$ for $p+q=n$.

\medskip

Another example, which will become important later, is the exceptional simple linear algebraic group $G_2(\C)=\mathrm{Aut}_{\C}(\mathbb{O}(\C))$. It has two $\R$-forms: the split form, whose group of $\R$-rational points is $G_{2(2)}$, the automorphism group of the real split octonions, and the compact form, whose group of $\R$-rational points is $\mathrm{Aut}_{\R}(\mathbb{O}(\R))$, which we denote by $G_2(\R)$.

\medskip

For later results, the notion of a \emph{reductive} group will play a central role. There are many equivalent characterizations, we again give one for the case where $k$ is algebraically closed.

\begin{definition}
A linear algebraic group $G$ is called \emph{reductive} if its largest smooth connected unipotent normal subgroup is trivial.
\end{definition}

\medskip

We now recall the notion of a regular (also called rational) representation.

\begin{definition}
Let $k$ be an algebraically closed field and $G$ a linear algebraic group. A representation $(V,\pi)$ is called \emph{regular} if $V$ is a finite-dimensional $k$-vector space and the map
\[
G\to k,\qquad g \mapsto \langle w,\pi(g)v\rangle
\]
is a regular morphism, i.e.\ belongs to $k[G]$ for all $v\in V,w \in V^*$.
\end{definition}

The way we defined a linear algebraic group here as Zariski closed $G<\mathrm{GL}_n(k)$,obviously the identity gives a regular representation on $k^n$ and thus for our case the above definition can be seen only as way to clarify what "acting regularly" means, later on.

\medskip

Let $(V,\pi)$ be a regular representation of a linear algebraic group $G$. The action of $G$ on $V$ induces a natural action on the coordinate ring $k[V]$ defined by
\[
(g\cdot f)(v):=f(\pi(g^{-1})v).
\]

\begin{definition}
Let $G$ be a linear algebraic group acting by means of a regular representation on a finite-dimensional $k$-vector space $V$. The \emph{invariant algebra} is the subalgebra
\[
k[V]^G=\{f\in k[V]\mid g\cdot f=f\ \text{for all } g\in G\}.
\]
\end{definition}

\medskip

The most important example for us arises when $G\leq\mathrm{Aut}_k(\A)$ for a finite-dimensional $k$-algebra $\A$. If $G$ is Zariski closed, then it is a linear algebraic group and its natural action on $\A$ is a regular representation. This action extends diagonally to $\A^m$, yielding a regular representation and the corresponding invariant algebra $k[\A^m]^G$.

\medskip

\begin{remark}
For a finite-dimensional $k$-algebra $\A$, we distinguish the group $\mathrm{Aut}_k(\A)$ of $k$-algebra automorphisms from the model-theoretic automorphism group $\mathrm{Aut}(\A)_{\mathcal L}$. The latter need not act $k$-linearly, whereas $\mathrm{Aut}_k(\A)$ is a Zariski-closed subgroup of $\mathrm{GL}_k(\A)$ and is therefore the natural group for invariant theory. This causes no problem for our model-theoretic arguments: whenever $G\leq\mathrm{Aut}_k(\A)$ preserves the additional symbols of $\mathcal L$, we have $G\leq\mathrm{Aut}(\A)_{\mathcal L}$, so $G$-orbits are contained in $\emptyset$-types by \Cref{type}.
\end{remark}
\medskip

\begin{remark}
It should be noted that linear algebraic groups, reductiveness and regular representations are definitions we introduced in a way, as they are often naturally given in our setting and also requirements for \Cref{fin} and \Cref{orbit}. The notion of an invariant algebra can, however, be defined in exactly the same way for a general field $k$ and any group $G$ acting linearly on a finite-dimensional $k$-vector space $V$, and we denote it again by $k[V]^G$.
\end{remark}

\medskip

If two points $v,w\in V$ lie in the same $G$-orbit, then clearly $f(v)=f(w)$ for all $f\in k[V]^G$. The following fundamental result describes when the converse holds.

\begin{theorem}\label{orbit}\cite{GIT}
Let $k$ be an algebraically closed field, $G$ a reductive linear algebraic group, and $(V,\pi)$ a regular representation. Then the invariant algebra separates orbit closures:
\[
\forall f\in k[V]^G:\ f(v)=f(w)
\quad\Longleftrightarrow\quad
\overline{\mathrm{Orb}(v,G)}\cap
\overline{\mathrm{Orb}(w,G)}\neq\emptyset.
\]
\end{theorem}

\medskip

A second fundamental property of reductive groups acting regularly on $V$ is finite generation of $k[V]^G$.

\begin{theorem}\label{fin}\cites{GIT,sym,haboush}
Let $k$ be an algebraically closed field, $G$ a reductive linear algebraic group, and $(V,\pi)$ a regular representation. Then the invariant algebra $k[V]^G$ is finitely generated.
\end{theorem}

\medskip

We now turn to compact real Lie groups acting linearly on real algebras. Every compact Lie group admits a faithful finite-dimensional real representation and can be treated as a compact linear group. However this generalization away from algebraic closed fields is somewhat more involved and in fact the real versions of \Cref{fin} and \Cref{orbit}, that we will show below, are simpler and historically older, dating back to Hilbert. In contrast to the results discussed above, we will state the results for $\R$ instead of general real closed fields, as this is how they can be found in the literature, and in fact may rely on completeness type arguments, not valid for general real closed fields.

\begin{theorem}\label{hil}\cite{gol}
Let $V$ be a finite-dimensional $\R$-vector space and $G$ a compact Lie group acting linearly on $V$. Then the invariant algebra $\R[V]^G$ is finitely generated.
\end{theorem}

\medskip

The next theorem, in the form we use it, reduces to an application of the Stone-Weierstrass theorem.

\begin{theorem}\label{schwarz}\cites{bredon,lie}
Let $V$ be a finite-dimensional $\R$-vector space and $G$ a compact Lie group acting linearly on $V^m$. Let $p_1,\dots,p_n$ be generators of the invariant algebra $\R[V^m]^G$. Then the orbit space $V^m/G$ is locally compact and Hausdorff, and the, by means of $p_i$, induced map
\[
p:V^m/G\to\R^n
\]
is an embedding. In particular, $p$ separates the $G$-orbits on $V^m$.
\end{theorem}

\medskip

The definitions above are only a glimpse of the very rich and old field of classical and geometric invariant theory. We compare this to the notion of \textit{complete invariant} we have given in our previous paper, which is much more "hands on" and will therefore remain useful for our results but is also devoid of any theory.  

\begin{definition}
Given an equivalence relation $\sim$ on a set $X$, a function $f:X\to Y$ is called a \emph{complete invariant} if
\[
\forall a,b\in X:\quad a\sim b \iff f(a)=f(b).
\]
If $f$ and $Y$ decompose compatibly as direct products, one also speaks of a \emph{complete set of invariants}.
\end{definition}

\section{Results}

For sake of completeness, we first restate here the general method introduced in~\cite{constructive} for transferring quantifier elimination.

Let $\mathcal{L}$ and $\mathcal{G}$ be formal languages and $A$ and $B$ structures over $\mathcal{L}$ and $\mathcal{G}$, respectively. Assume that for every $m\in\N$ there exists a function
\[
f_m:B^m\to A^{d(m)}
\]
with the following properties:

\begin{enumerate}
\item[(i)] For any definable set $D\subseteq B^m$, the set $f_m(D)$ is definable in $A^{d(m)}$.
\item[(ii)] For any quantifier-free definable set $E\subseteq A^{d(m)}$, the set $f_m^{-1}(E)\subseteq B^m$ is quantifier-free definable.
\item[(iii)] For any $a\in A^{d(m)}$ and any definable set $D\subseteq B^m$, we either have $f_m^{-1}(a)\subseteq D$ or $f_m^{-1}(a)\cap D=\emptyset$.
\end{enumerate}

We call conditions (i) and (ii) \emph{constructive} if there is an algorithm transforming definitions of sets into definitions of their image/preimage, as required.

\medskip

Properties (i) and (ii) allow us to transfer definable sets back and forth via $f_m$. Property (iii) basically says that $f_m$ is "injective up to definability".

\begin{theorem}\label{theo}
Let $A,B$ and $(f_m)_{m\in\N}$ be as described above. Then, if $A$ admits quantifier elimination, so does $B$. Furthermore, if (i) and (ii) are constructive and there exists a constructive quantifier-elimination algorithm for $A$, then there is also a constructive algorithm for quantifier elimination in $B$.
\end{theorem}

\begin{proof}
Let $D\subseteq B^m$ be definable. By (i), $f_m(D)$ is definable in $A^{d(m)}$. Since $A$ admits quantifier elimination, $f_m(D)$ is quantifier-free definable and so by (ii), $f_m^{-1}(f_m(D))$ is also quantifier-free definable in $B^m$. By (iii), we have $f_m^{-1}(f_m(D))=D$, which proves the first statement. The statement about constructiveness also follows directly from this argument.
\end{proof}
\medskip
Let in the following $k$ always denote a real closed field.
Before applying \Cref{theo} to $\mathbb{H}(k)$, we need the following lemma, which provides us with a complete invariant for the diagonal action of $SO_3(k)\cong\mathrm{Aut}_k(\mathbb{H}(k))$ on $\mathbb{H}(k)^m$.

\begin{lemma}\label{quainv}
Define
\[
L_1:\mathbb{H}(k)^m\to k^{m}, 
\qquad 
\underline{v}\mapsto (\operatorname{Re}(v_i))_{i\le m},
\]
\[
L_2:\mathbb{H}(k)^m\to k^{m^2}, 
\qquad 
\underline{v}\mapsto 
\bigl(\operatorname{Re}((v_i-\overline{v_i})(v_j-\overline{v_j}))\bigr)_{i,j\leq m},
\]
\[
L_3:\mathbb{H}(k)^m\to k^{m^3}, 
\qquad 
\underline{v}\mapsto 
\bigl(\operatorname{Re}((v_i-\overline{v_i})(v_j-\overline{v_j})(v_k-\overline{v_k}))\bigr)_{i,j,k\leq m}.
\]

Then the orbits of the diagonal action of $SO_3(k)$ on $\mathbb{H}(k)^m$ are separated by $L_1,L_2,L_3$, i.e.
\[
L_1(\underline{v})=L_1(\underline{w})
\land
L_2(\underline{v})=L_2(\underline{w})
\land
L_3(\underline{v})=L_3(\underline{w})
\]
if and only if
\[
\exists R\in SO_3(k):\quad R\cdot v_i=w_i \text{ for all } i\le m.
\]
\end{lemma}

\begin{proof}
For simplicity we assume $k=\R$. Since $SO_3(k)$ is definable (even in $\mathcal{F}$ without $\leq$), the statement can be expressed in first-order logic, and thus the general case follows by Tarski's transfer principle.

The action of $SO_3(\R)$ is defined by
\[
R\cdot v
=
\operatorname{Re}(v)
+
\langle R\operatorname{Im}(v)^T,(\mathbf{i},\mathbf{j},\mathbf{k})^T\rangle.
\]
It is clear that $R\cdot\overline{v}=\overline{R\cdot v}$ and $\operatorname{Re}(R\cdot v)=\operatorname{Re}(v)$. This makes the “$\Leftarrow$” direction immediate.

For the “$\Rightarrow$” direction, note that $L_1(\underline{v})=L_1(\underline{w})$ gives the clearly necessary condition for the real parts to coincide. Thus the problem reduces to finding a rotation $R\in SO_3(\R)$ simultaneously transforming $\operatorname{Im}(v_i)$ to $\operatorname{Im}(w_i)$.

For $L_2$ and $L_3$ we observe that by elementary quaternion algebra
\[
\operatorname{Re}((v-\overline{v})(w-\overline{w}))
=
-4\langle\operatorname{Im}(v)^T,\operatorname{Im}(w)^T\rangle
\]
and
\[
\operatorname{Re}((v-\overline{v})(w-\overline{w})(z-\overline{z}))
=
-8\langle
\operatorname{Im}(v)^T,
\operatorname{Im}(w)^T\times\operatorname{Im}(z)^T
\rangle.
\]

It is well known that if $m$ points $\underline{p}\in\R^3$ have the same pairwise angles and lengths as $m$ other points $\underline{q}\in\R^3$, then there exists a transformation $R\in O_3(\R)$ with $Rp_i=q_i$ for all $i\le m$.

To ensure that $R$ can be chosen orientation-preserving (i.e.\ contained in $SO_3(\R)$), the condition
\[
\langle p_i,(p_j\times p_k)\rangle
=
\langle q_i,(q_j\times q_k)\rangle
\]
for any triple $p_i,p_j,p_k$ in general position is sufficient. In the case where all points lie in a proper subspace, $R$ can also trivially be chosen to be orientation preserving. This completes the proof.
\end{proof}

\begin{corollary}\label{quat}
The $(\mathcal{F},\leq,\bar \ )$-structure $\mathbb{H}(k)$ admits quantifier elimination.
\end{corollary}

\begin{proof}
We apply \Cref{theo} to the $(\mathcal{F},\leq,\bar \ )$-structure $\mathbb{H}(k)$ and the $(\mathcal{F},\leq)$-structure $k$, which admits quantifier elimination.

For this define
$
d(m)=m+m^2+m^3$
and
\[
f_m:\mathbb{H}(k)^m\to k^{d(m)},
\qquad
\underline{v}\mapsto
L_1(\underline{v})\oplus L_2(\underline{v})\oplus L_3(\underline{v}).
\]
where we define $L_i$ as in \Cref{quainv}.
We have to verify conditions (i)–(iii).

\medskip

(i) is clear. Just write both the defining formula and the definition of $f_m$ (which as polynomial map $k^{4m} \to k^{d(m)}$ only has rational coefficients) down in the vector entries and use existential quantifiers to obtain the image.

\medskip

For (ii), suppose $\phi(x_1,\dots,x_{d(m)})$ is a quantifier-free $(\mathcal{F},\leq)$-formula defining the set $D\subset k^{d(m)}$. Since
\[
\operatorname{Re}(v)=\iota^{-1}\!\left(\frac{1}{2}(v+\overline{v})\right),
\]
each coordinate of $\iota^{(d(m))}(f_m(y_1,...,y_m))$ is a term in $(\mathcal{F},\bar \ )$ (we can clearly treat $\frac{1}{2}$ like a constant symbol). Substituting these into the atomic subformulas of $\phi$ yields a quantifier-free formula, which we write symbolically as
\[
\phi\bigl(\iota(f_m(y_1,\dots,y_m)_1),...,\iota(f_m(y_1,\dots,y_m)_{d(m)})\bigr)
\]
and which defines $f_m^{-1}(D)$ in $(\mathcal{F},\leq,\bar \ )$.

\medskip

For (iii), recall that $f_m$ is a complete invariant for the diagonal action of $SO_3(k)$ on $\mathbb{H}(k)^m$ by \Cref{quainv}. Since for any $R\in SO_3(k)$ we have $R\cdot\overline{v}=\overline{R\cdot v}$ and $\operatorname{Re}(R\cdot v)=\operatorname{Re}(v)$, it follows that
\[
SO_3(k)
=
\mathrm{Aut}_k(\mathbb{H}(k))
\leq
\mathrm{Aut}(\mathbb{H}(k))_{\mathcal{F}}=
\mathrm{Aut}(\mathbb{H}(k))_{(\mathcal{F},\leq,\bar \ )}.
\]
Hence by \Cref{type} the fibers of $f_m$ have constant $\emptyset$-type, in particular, they cannot be split by definable sets.
\end{proof}
\medskip

We can repeat the basic proof idea for the octonions $\mathbb{O}(k)$, almost the same. But first, we again need to determine a complete invariant, which turns out a little more complicated than in the quaternion case.

\begin{lemma}\label{ocinv}
Define
\[
L:\mathbb{O}(k)^m\to k^{|I(m)|},
\qquad
\underline{v}\mapsto
\bigl(\operatorname{Re}(\omega_i(v_1,\dots,v_m))\bigr)_{i\in I(m)},
\]
where $I(m)$ denotes the set of all possible (nonassociative) words in $m$ letters of length $\le 4$.

Then the orbits of the diagonal action of $\mathrm{Aut}_k(\mathbb{O}(k))$ on $\mathbb{O}(k)^m$ are separated by $L$, i.e.
\[
L(\underline{v})=L(\underline{w})
\iff
\exists R\in \mathrm{Aut}_k(\mathbb{O}(k)):\quad R\cdot v_i=w_i \text{ for all } i\le m.
\]
\end{lemma}

\begin{proof}
Again we first assume $k=\R$.

We have
\[
\mathrm{Aut}_{\R}(\mathbb{O}(\R))=G_2(\R),
\]
which is a real compact Lie group and can be defined as the stabilizer of the action $SO_7(\R)\curvearrowright \mathbb{O}(\R)$ with respect to the $3$-form $
\phi
=
e^{123}+e^{145}+e^{167}+e^{246}
-e^{257}-e^{347}-e^{356},
$ where $e^{ijk}$ denotes the exterior product of the corresponding dual basis vectors of $\R^7$.

Its Lie group complexification is the connected and reductive linear algebraic group
\[
G_2(\C)=\mathrm{Aut}_{\C}(\mathbb{O}(\C)).
\]
$G_2(\R)$ is a Zariski dense $\mathbb{R}$-form in $G_2(\C)$ and it naturally acts on $\mathbb{O}(\C)$ by $g\otimes \mathrm{id}_{\C}$ (see \cite{sym} for density of real forms or also~\cites{borel,grot} for more general statements regarding unirationality).

Now it has been shown in \cite{schwarzoct} (see also section 2.4. in \cite{alginv} for the explicit statement as is used here) that $
\C[\mathbb{O}(\C)^m]^{G_2(\C)}
$ is generated by the forms
\[
\operatorname{Re}(\omega(x_1,\dots,x_m)),
\]
where $\omega$ ranges over all nonassociative words of length $\le 4$. Note that all those  forms, written as homogeneous polynomials in the vector entries, have real valued coefficients.

Again the “$\Leftarrow$” direction is clear, since $G_2(\R)\curvearrowright \mathbb{O}(\R)^m$ fixes the real parts.

For the “$\Rightarrow$” direction, assume $L(\underline{v})=L(\underline{w})$. Passing to the complexification $\mathbb{O}(\C)^m=(\mathbb{O}(\R)\otimes_\R \C)^m$, it then follows from \cite{schwarzoct} that
\[
f(v_1\otimes 1_\C,\dots,v_m\otimes 1_\C)
=
f(w_1\otimes 1_\C,\dots,w_m\otimes 1_\C)
\quad
\text{for all } f\in\C[\mathbb{O}(\C)^m]^{G_2(\C)}.
\]

But by Zariski density, every $G_2(\R)$-invariant function in $\C[\mathbb{O}(\C)^m]$ is also $G_2(\C)$-invariant (the evaluation function $F_v(g):=f(g\cdot v)$ for arbitrary fixed $v\in \mathbb{O}(\C)^m$ and $f\in\C[\mathbb{O}(\C)^m]^{G_2(\R)}$ is clearly regular) and we get
\[
f(v_1\otimes 1_\C,\dots,v_m\otimes 1_\C)
=
f(w_1\otimes 1_\C,\dots,w_m\otimes 1_\C)
\quad
\text{for all } f\in\R[\mathbb{O}(\C)^m]^{G_2(\R)}\otimes_\R 1_\C \subset \C[\mathbb{O}(\C)^m]^{G_2(\R)}
\]

So by restricting back to the real subspace $\mathbb{O}(\R)\otimes 1_\C$, we can use compactness of  $G_2(\R)$ to apply \Cref{schwarz}:
\[
\exists R\in G_2(\R):\quad (R\otimes \mathrm{id}_\C)\cdot (v_i \otimes 1_\C)=w_i\otimes 1_\C \text{ for all } i\le m.
\]
and the desired result follows.\\
\\
Again the result for general real closed fields follows by Tarski's transfer principle as $G_2(k)$ is a definable subset of $k^{49}$, as can be seen for example by its definition as stabilizer subgroup.
\end{proof}

\begin{corollary}\label{Octonions}
The $(\mathcal{F},\leq,\bar \ )$-structure $\mathbb{O}(k)$ admits quantifier elimination.
\end{corollary}

\begin{proof}
We proceed exactly as in the proof of \Cref{quat}, mutatis mutandis, using
\[
G_2(k)=\mathrm{Aut}_k(\mathbb{O}(k))
\leq
\mathrm{Aut}(\mathbb{O}(k))_{(\mathcal{F},\leq,\bar \ )},
\]
the fact that $L$ from \Cref{ocinv} seen as polynomial map $k^{8m}\to k^{|I(m)|}$ only has rational coefficients,
and finally again using the function
\[
\operatorname{Re}(v)
=
\iota^{-1}\!\left(\frac{1}{2}(v+\overline{v})\right).
\]

\end{proof}

\begin{remark}
For both quaternions and octonions the symbol $\bar \ $ was needed to be able to express the invariant in our extended language. Note that conjugation, the scalar-part map $x\mapsto\operatorname{Re}(x)1$, the imaginary-part map $x\mapsto x-\operatorname{Re}(x)1$, and $N$ are quantifier-free interdefinable, so the precise choice of the additional symbol can be handled flexibly.
On the other hand it was essential for our proof to extend $\mathcal{F}$ by $\leq$ and interpret it as the partial order on the center. This had to be done to ensure condition (ii) of \Cref{theo} and indeed, the symbol $\bar \ $ allows for defining a copy of $k$ in the $(\mathcal{F},\bar \ )$-structure $\mathbb{H}(k)$ or $\mathbb{O}(k)$ without quantifiers. Thus formulas like $$\exists y :\text{Re}(x)=\text{Re}(y)\text{Re}(y)$$ suggest the necessity for the addition of the order symbol $\leq$. \\

\end{remark}
\medskip

For the special cases of $\mathbb{H}(k)$ and $\mathbb{O}(k)$, we explicitly determined complete invariants in order to state the result for a concrete language extension of $\mathcal{F}$ and also to preserve constructiveness. If we give up on these two features, we can use the general invariant theory developed above to obtain more general statements.

\begin{proposition}\label{QEInv}
Let $k$ be a real closed field in the language $(\mathcal{F},\leq)$ and $\A$ an $n$-dimensional $k$-algebra. Assume that
\[
G\leq\mathrm{Aut}_k(\A)
\]
is a $(\mathcal{F},\leq)$-definably compact subset of $k^{n^2}$ with respect to the order topology. Let $S_m$ be a finite generating set of $k[\A^m]^G$ for each $m\in \N$. Choose $(\mathcal{F},\leq,E)$ to be an extended language such that the interpretation of each $e\in E$ for $\A$ lies in $\iota(k[\A^{\text{arity}(e)}]^G)$ and for all $s\in S_m$, $\iota(s)$ corresponds to the evaluation of a $(\mathcal{F},\leq,E)$-term. Assume furthermore that, after identifying $\A$ with $k^n$ via a $k$-basis beginning with $1_\A$, the expanded $(\mathcal{F},\leq,E)$-structure $\A$ is $\emptyset$-definable in the $(\mathcal{F},\leq)$-structure $k$. Then the $(\mathcal{F},\leq,E)$-structure $\A$ admits quantifier elimination.
\end{proposition}

\begin{proof}
Define
\[
f_m:\A^m\to k^{|S_m|},
\qquad
f_m(\underline{a})_i=s_i(\underline{a}).
\]

We want to apply \Cref{theo}. Since $k$ is real closed, it admits quantifier elimination in $(\mathcal{F},\leq)$, so it only remains to verify (i)–(iii).

\medskip

Condition (i) follows from the assumed parameter-free definability of the expanded structure in $k$. Indeed, every formula defining $D\subseteq\A^m$ can be translated coordinatewise into a formula over $k$. Moreover, since the scalar lifts of the components of $f_m$ are terms in the expanded language, the graph of $f_m$ is definable in $k$. Hence $f_m(D)$ is definable by existentially quantifying over the coordinates of the elements of $D$.

\medskip

(ii) follows by assumption on the extension $E$. If $\phi(x_1,\dots,x_{|S_m|})$ is a quantifier-free $(\mathcal{F},\leq)$-formula defining $D\subset k^{|S_m|}$, then
\[
\phi\bigl(\iota(s_1(y_1,\dots,y_m)),\dots,\iota(s_{|S_m|}(y_1,\dots,y_m))\bigr)
\]
is a quantifier-free $(\mathcal{F},\leq,E)$-formula defining $f_m^{-1}(D)$.

\medskip

For (iii), first assume $k=\R$. By \cite{comp}, in $o$-minimal structures definable compact sets can equivalently be defined as definable, closed and bounded sets. These are first-order definable properties in $(\mathcal{F},\leq)$ and therefore can be transferred to classical compactness on $\R$ by Heine-Borel.

Then, since $G$ is a compact subgroup of $\mathrm{Aut}_k(\A)\leq\mathrm{GL}_k(\A)$, it is a Lie group. Thus by \Cref{schwarz} $k[\A^m]^G$ separates the orbits of the diagonal action $G\curvearrowright \A^m$. Since $(s_i)$ generate $k[\A^m]^G$ we get the following relation:$$f_m(\underline{a})=f_m(\underline{b})\Leftrightarrow i(\underline{a})=i(\underline{b}) \forall i\in k[\A^m]^G\Leftrightarrow \exists g\in G:g\underline{a}=\underline{b}$$
Since the functions interpreted by $E$ are $G$-invariant and scalar-valued, every $g\in G$ is an automorphism of the expanded structure, that is,
\[
G\leq\mathrm{Aut}(\A)_{(\mathcal{F},\leq,E)}.
\]
Hence \Cref{type} implies that $\underline{a}$ and $\underline{b}$ have the same $\emptyset$-type in $(\mathcal{F},\leq,E)$. By definability of $G$ the statement for general real closed $k$ follows by Tarski's transfer principle, which finishes the proof of $(iii)$.
\end{proof}

\begin{remark}
The conditions for the language extension $E$ in \Cref{QEInv} seem somewhat far fetched. However, note that so called \textit{first fundamental theorems} (see \cite{sym}) give explicit descriptions of generators for many invariant algebras of $k$-algebras, and their corresponding automorphism groups (or subgroups thereof) by means of \textit{tensor forms}. These are usually stated for the algebraically closed case, however, it should be possible to often use Zariski denseness arguments, as we did in the proof of \Cref{ocinv}, to infer back to the real closed case.
\end{remark}

\medskip

What about algebras $\A$ over an algebraically closed field (in the following always denoted $k$)? Compactness of the acting group was essential in the real closed case. Comparing with \Cref{orbit}, we see that for a general reductive linear algebraic group $G$, $k[V]^G$ separates only orbit closures, which need not separate $\emptyset$-types. Over $k=\C$, a natural remedy would again be to require $G$ to be compact in the Euclidean topology. However, a compact complex affine algebraic group must already be finite, restricting us quite a bit.

\begin{proposition}\label{QEInvalg}
Let $k$ be an algebraically closed field in the language $\mathcal{F}$ and $\A$ an $n$-dimensional $k$-algebra. Assume $G\leq\mathrm{Aut}_k(\A)$ is a finite linear algebraic group. Let $S_m$ be a finite generating set of $k[\A^m]^G$ for each $m\in \N$, and let $(\mathcal{F},E)$ be an extended language such that the interpretation of each $e \in E$ in $\A$ lies in $\iota(k[\A^{\text{arity}(e)}]^G)$ and for all $s\in S_m$, $\iota(s)$ corresponds to evaluation of a $(\mathcal{F},E)$-term. Assume furthermore that, after identifying $\A$ with $k^n$ via a $k$-basis beginning with $1_\A$, the expanded $(\mathcal{F},E)$-structure $\A$ is $\emptyset$-definable in the $\mathcal{F}$-structure $k$. Then the $(\mathcal{F},E)$-structure $\A$ admits quantifier elimination.
\end{proposition}

\begin{proof}
The sets $S_m$ are finite by assumption (note that such an $S_m$ always exist by reductiveness of $G$ (see below) and \Cref{fin}), so define
\[
f_m:\A^m\to k^{|S_m|},
\qquad
f_m(\underline{a})_i=s_i(\underline{a}).
\]

Since $k$ is algebraically closed, it admits quantifier elimination. Conditions (i) and (ii) of \Cref{theo} are verified exactly as in \Cref{QEInv}.

For (iii), note that since a finite group is discrete it is also by definition reductive and obviously has closed orbits. Hence by \Cref{orbit}, $k[\A^m]^G$ separates $G$-orbits. As before,
\[
f_m(\underline{a})=f_m(\underline{b})
\iff
\exists g\in G:\ g\underline{a}=\underline{b}.
\]
Since the functions interpreted by $E$ are $G$-invariant and scalar-valued, we have
\[
G\leq\mathrm{Aut}(\A)_{(\mathcal{F},E)}.
\]
Hence \Cref{type} implies that $\underline{a}$ and $\underline{b}$ have the same $\emptyset$-type in $(\mathcal{F},E)$.
\end{proof}

\medskip

\begin{remark}
The use of \Cref{orbit} in the above proof is in fact slightly overkill, as there are more elementary ways to see separation of orbits in the finite case. We still state it like this, in order to remain in the invariant theoretic spirit of the rest of this article.   
\end{remark}

\medskip

For a finite-dimensional associative $k$-algebra $\A$ ($k$ again algebraically closed), finiteness of $\mathrm{Aut}_k(\A)$ forces $\A$ to be commutative and semisimple: inner automorphisms rule out noncommutativity, while a nonzero Jacobson radical would imply infinitely many automorphisms due to the existence of nontrivial derivations. Hence, by Wedderburn's Theorem,this would imply $\A\cong k^n$. However finite subgroups of $\mathrm{Aut}_k(\A)$ exist and their invariant algebras can be studied (\cites{compu,fin}).
For non-associative algebras, $\mathrm{Aut}_k(\A)$ can even itself be finite (and non-trivial). For example in \cite{low}, all simple, two dimensional $k$-algebras have been classified and their corresponding automorphism groups (some of which are finite) and invariant algebras calculated explicitly.

\medskip
 For infinite reductive linear algebraic groups, the situation becomes much more hopeless. Conditions on specific points $x$ to give closed orbits exist (see for example \cites{KN,birkes}), but these are hard to apply for the whole space simultaneously.
 For example, if $G$ is defined over $\R$ with a regular representation on a real vector space $V$ and $G(\R)$ is compact, it follows from \cite{birkes} that the $G(\R)$-orbits in $V_\R$ are closed. However, this generally still fails to extend to all complex points of $V_\C$.\\
 In fact, the Hilbert--Mumford stability criterion \cite{GIT} implies that if $G$ contains a one-parameter subgroup $\lambda:\C^{\times}\to G$ acting nontrivially on a complex vector space $V$, then there exists some $0\neq v\in V$ such that $0\in \overline{\operatorname{Orb}(v,G)}$.

\section{Acknowledgements}
I want to greatly thank my supervisor Prof. Tim Netzer, for many helpful comments on a preliminary version of this article.\\
\\
I also gratefully acknowledge funding by the Austrian Academy of Sciences
(OAW), through a DOC-Scholarship.

\end{document}